\documentclass[12pt,letter]{amsart}
\usepackage[latin1]{inputenc}
\usepackage{tikz}
\usetikzlibrary{arrows.meta}
\usetikzlibrary{bending}
\usepackage[all,color]{xy}
\usepackage{mathrsfs,import,stmaryrd,wasysym,amsmath,amssymb,amsthm,amscd,amsxtra,amsfonts,mathrsfs,graphicx,enumerate,bm,slashed,array,setspace,enumitem,minibox,booktabs,rotating,multirow,adjustbox,hhline,arydshln,placeins,tikz-feynman,hyperref,verbatim,xcolor}

\numberwithin{equation}{section} 

\input xy
\xyoption{all}
\usetikzlibrary{shapes,arrows}
\usetikzlibrary{positioning}
\usetikzlibrary{calc}

\input xy
\xyoption{all}

\theoremstyle{plain}
\newtheorem{Theorem}{Theorem}[section]
\newtheorem{Lemma}[Theorem]{Lemma}

\newtheorem{Proposition}[Theorem]{Proposition}
\newtheorem*{Theorem*}{Theorem}

\theoremstyle{definition} 
\newtheorem{Remark}[Theorem]{Remark}
\newtheorem{Example}[Theorem]{Example}
\newtheorem{Definition}[Theorem]{Definition} 
\newtheorem{Notation}[Theorem]{Notation}
\newtheorem*{Definition*}{Definition}

\theoremstyle{definition} 

\DeclareMathOperator{\pd}{pd}
\DeclareMathOperator{\Hom}{Hom}
\DeclareMathOperator{\Ext}{Ext}
\DeclareMathOperator{\Max}{Max}
\DeclareMathOperator{\ann}{ann}

\newcommand*{\overbar}[1]{\mkern 1.5mu\overline{\mkern-1.5mu#1\mkern-1.5mu}\mkern 1.5mu}
\advance\evensidemargin-.5in
\advance\oddsidemargin-.5in
\advance\textwidth1in

\definecolor{darkgreen}{rgb}{0.0, 0.5, 0.0}
\definecolor{darkred}{rgb}{0.7, 0.11, 0.11}

\begin{document}
\author{Karin Baur}
\address{School of Mathematics, University of Leeds, Leeds, LS2 9JT, United Kingdom}
\address{Ruhr-Universit\"at Bochum, 
Universit\"atsstrasse 150, 
D-44780 Bochum, Germany
}
\email{karin.baur@rub.de}

\author{Charlie Beil}
 \address{Institut f\"ur Mathematik und Wissenschaftliches Rechnen, Universit\"at Graz, Heinrichstrasse 36, 8010 Graz, Austria.}
 \email{charles.beil@uni-graz.at}
 \title[Global dimensions of local geodesic ghor algebras]{Global dimensions of\\ local geodesic ghor algebras}
 \date{}

\begin{abstract}
A ghor algebra is a path algebra with relations of a dimer quiver in a compact surface. 
We show that the global dimension of any cyclic localization of a geodesic ghor algebra on a genus $g \geq 1$ surface is bounded above by $2g+1$.
This number coincides with the Krull dimension of the center of the ghor algebra.
We further show that the bound is an equality if and only if the point of localization sits over the noetherian locus of the center.
\end{abstract}

\maketitle

\tableofcontents

\section{Introduction}

Ghor algebras are a type of noncommutative algebra constructed from oriented graphs on surfaces, called dimer quivers.
In this article we establish new connections between their representation theory, the topology of surfaces, and nonnoetherian algebraic geometry.

Algebras constructed from dimer quivers were introduced in physics in 2005 \cite{HK05, FHVWK06}.
Since then, they have found numerous applications to both mathematics and physics; see for example \cite{DKS24,INU24,CKP24,BKM16,BCQ15,MR10} and references therein. 

A \textit{dimer quiver} $Q$ is a quiver that embeds in a surface $\Sigma$, such that each connected component of $\Sigma \setminus Q$ is simply connected and bounded by an oriented cycle, called a \textit{unit cycle}.
The \textit{dimer algebra} of $Q$ is then the quotient $kQ/I$ of the path algebra $kQ$ by the ideal
\begin{equation} \label{I}
I = \left\langle p - q \ | \ \exists a \in Q_1 \text{ such that } pa, qa \text{ are unit cycles} \right\rangle \subset kQ,
\end{equation}
where $p,q$ are paths and $Q_1$ is the set of arrows of $Q$. 
Dimer algebras on a torus have exceptional algebraic and homological properties (see for example \cite{B21,Br12,D11,BCQ15}).
In particular, let $A = kQ/I$ be a noetherian dimer algebra on a torus, let $R$ be its center, and let $\mathfrak{m}$ be the maximal ideal at the origin of the algebraic variety $\operatorname{Max}R$ of $R$.\footnote{A dimer algebra on a torus is noetherian if and only if it is cancellative \cite{B21}, if and only if it satisfies certain combinatorial conditions called `consistency' \cite{D11}, if and only if it is Calabi-Yau \cite{Br12}.}
Then the Krull dimension of $R$ and the global dimension of the localization $A_{\mathfrak{m}} := A \otimes_R R_{\mathfrak{m}}$ coincide,
\begin{equation} \label{gldim = dim}
\operatorname{gldim} A_{\mathfrak{m}} = 3 = \dim R = \operatorname{ht}(\mathfrak{m}). 
\end{equation}
However, this remarkable relationship between global dimension and Krull dimension disappears on higher genus surfaces.

Indeed, if $kQ/I$ is a dimer algebra on any 
surface of genus $\geq 2$, then its center is the polynomial ring in one variable, $Z(kQ/I) \cong k[\sigma]$ (Lemma \ref{lm:no dimer}). Thus, there can be no interesting relationship of the form (\ref{gldim = dim}).
We therefore desire a new kind of algebra, constructed from dimer quivers, that is able to maintain such a relationship on higher genus surfaces. 
The purpose of this article is to identify a class of algebras with this property.

The algebras we will consider are quotients of dimer algebras called `ghor algebras', introduced in \cite{B24}. 
A ghor algebra $A$ is a path algebra of a dimer quiver $Q$ on a surface $\Sigma$, modulo relations determined by the perfect matchings of $Q$; precise definitions will be given in Section~\ref{geodesic ghor algebra section}.
A ghor algebra is intimately related to the topology of the surface its quivers is embedded in.
For example, homotopic paths in $Q$ are identified in $A$, up to unit cycles. 

In this article we will restrict our attention to ghor algebras that are geodesic, meaning that at each vertex of $Q$ there is a `straight line path' along each direction of $\Sigma$.
This is a very natural property as it generalizes the key notion of cancellativity (or consistency) of dimer algebras on a torus to surfaces of arbitrary genus. 

Let $A$ be a geodesic ghor algebra, and let $R$ be its center.
An essential role in our study will be played by the cycle algebra $S$ of $A$, defined in (\ref{cycle algebra def}).
This is a noetherian subalgebra of $k[\mathcal{P}]$ generated by the cycles in $Q$, where $\mathcal{P}$ is the set of perfect matchings of $Q$.
The center $R$, in turn, is a subalgebra of $S$ \cite[Corollary 3.15]{BB21a}.

If the surface is a torus, then $R = S$ \cite[Theorem 1.1]{B21}.
On higher genus surfaces, however, $R$ is almost never noetherian, and so $R$ is properly contained in $S$.
Nevertheless, the geometry of $R$ can be understood from the geometry of $S$ \cite[Theorem 4.7]{BB21a}.\footnote{For this statement, the ground field $k$ is assumed to be uncountable.} For this reason, $S$ is called a depiction of $R$.
Specifically, $\operatorname{Max}R$ may be identified with the algebraic variety $\operatorname{Max}S$, except that some positive dimensional subvarieties of $\operatorname{Max}S$ become closed points \cite{B16}.
The complement of these subvarieties is called the \textit{noetherian locus}, denoted $U_{S/R} \subset \operatorname{Max}S$, and is the open dense set where $R$ and $S$ locally coincide, 
\begin{equation} \label{noetherian locus}
U_{S/R} := \left\{ \mathfrak{n} \in \operatorname{Max}S \, | \, R_{\mathfrak{n} \cap R} = S_{\mathfrak{n}} \right\} = \left\{ \mathfrak{n} \in \operatorname{Max}S \, | \, R_{\mathfrak{n} \cap R} \text{ is noetherian} \right\}.
\end{equation}
Localized at points in this locus, $A$ becomes an endomorphism ring of a module over its center \cite[Theorem 4.15]{BB21a}.

Our main result is a bound on the global dimension in terms of the topology of the surface.\footnote{The algebra homomorphism $\eta$ in the statement of the theorem is defined in (\ref{eta map}) below.}

\begin{Theorem*} (Theorem \ref{main corollary})
Let $\Sigma$ be a smooth compact genus $g \geq 1$ surface.
Let $A = kQ/\ker \eta$ be a geodesic ghor algebra on $\Sigma$ with center $R$ and cycle algebra $S$. 
Then for each maximal ideal $\mathfrak{n} \in \operatorname{Max}S$, the global dimension of the cyclic localization $A_{\mathfrak{n}}$ is bounded above by the Krull dimensions of $R$ and $S$,
\begin{equation*}
\operatorname{gldim} A_{\mathfrak{n}} \leq \dim R = \dim S = \operatorname{rank} H_1(\Sigma) + 1 = 2g+1,
\end{equation*}
with equality if, and only if for $g \geq 2$, $\mathfrak{n}$ is in the noetherian locus $U_{S/R} \subset \operatorname{Max}S$.
\end{Theorem*}

We note that the theorem holds for any surface obtained from identifying the opposite sides, and all of the vertices, of a regular polygon with an even number of sides.
In particular, it also holds for compact surfaces with two points identified. 

We conjecture that the strict inequality $\operatorname{gldim}A_{\mathfrak{n}} < \dim R$ over the nonnoetherian locus $\operatorname{Max}S \setminus U_{S/R}$ generalizes the relation (\ref{gldim = dim}) to 
\begin{equation*}
\operatorname{gldim}A_{\mathfrak{n}} = \operatorname{ght}(\mathfrak{n} \cap R)\le 2g+1,
\end{equation*}
where $\operatorname{ght}(\mathfrak{n} \cap R)$ is the so-called `geometric height' of the maximal ideal $\mathfrak{n} \cap R$ of $R$. 
This new notion of height, introduced in \cite{B16}, corresponds to the codimension of the corresponding positive dimensional point in $\operatorname{Max}R$.\footnote{Specifically, for $\mathfrak{p} \in \operatorname{Spec}R$, $\operatorname{ght}(\mathfrak{p})$ is the infimum of heights of ideals sitting over $\mathfrak{p}$ over all depictions of $R$.}

\section{Preliminaries:\ geodesic ghor algebras} \label{geodesic ghor algebra section}

Throughout, $k$ is an algebraically closed field. 

Let $\Sigma$ be a compact orientable surface of genus $g \geq 0$, either smooth or pinched at one point (i.e., with two points identified).
We allow pinching so that we have a surface of Euler characteristic $\chi(\Sigma) = -n$ for each positive integer $n \in \mathbb{N}$, rather than just the even integers that come from smooth surfaces, where $\chi(\Sigma) = 2 - 2g$.
This generality is important for ghor algebras because the Krull dimensions of their centers are directly related to the Euler characteristic. 
The surface $\Sigma$ may be obtained by identifying the opposite sides, and all of the vertices, of a convex $2N$-gon $P$.
If $\Sigma$ is smooth, then $N = 2g$ is even, and if $\Sigma$ is pinched, then $N = 2g+1$ is odd.\footnote{If we only require the opposite sides of $P$ to be identified, but not all the vertices, then all the vertices will nevertheless be automatically identified if $N$ is even, but not if $N$ is odd. 
In the case $N$ is odd, that is, $N = 2g+1$, $\Sigma$ will be a smooth genus $g$ surface just as in the even case.}
($P$ is a so-called fundamental domain or fundamental polygon of $\Sigma$.) 
 
Let $Q$ be a dimer quiver embedded in $\Sigma$.
A \textit{perfect matching} of $Q$ is a set of arrows $x \subset Q_1$ such that each unit cycle contains precisely one arrow in $x$.
Furthermore, a perfect matching $x \in \mathcal{P}$ is \textit{simple} if, given any pair of vertices $i,j \in Q_0$, there is a path from $i$ to $j$ which avoids all arrows of $x$ (that is, $Q \setminus x$ is the support of a simple module of dimension vector $1^{Q_0}$). 
Example~\ref{ex:first-example} gives some (simple) perfect matchings. 
We denote by $\mathcal{P}$, $\mathcal{S}$ the set of perfect and simple matchings of $Q$ respectively, and by $k[\mathcal{P}]$, $k[\mathcal{S}]$ the polynomial rings with generating sets $\mathcal{P}$, $\mathcal{S}$.

\begin{Definition} \cite{B24} 
Let $Q$ be a dimer quiver embedded in a surface $\Sigma$.
Denote by $n := |Q_0|$ the number of vertices of $Q$, and by $e_{ij} \in M_n(k)$ the $n \times n$ matrix with entry $1$ in the $ij$-th slot and zeros elsewhere.
Consider the algebra homomorphism $\eta: kQ \to M_n(k[\mathcal{P}])$ 
defined on the vertices $i \in Q_0$ and arrows $a \in Q_1$ by
\begin{equation} \label{eta map}
\begin{array}{c}
\eta(e_i) = e_{ii}, \ \ \ \ \ \ \ \ \eta(a) = e_{\operatorname{h}(a),\operatorname{t}(a)} \prod_{\substack{x \in \mathcal{P} :\\ x \ni a}} x,\\
\end{array}
\end{equation}
and extended multiplicatively and $k$-linearly to $kQ$.
The \textit{ghor algebra} of $Q$ is the quotient\footnote{The word `ghor' is Klingon for surface.}
\begin{equation*}
A := kQ/\operatorname{ker}\eta.
\end{equation*}
\end{Definition}

Using simple matchings, we define the algebra homomorphism $\tau: kQ \to M_n(k[\mathcal{S}])$ similar to $\eta$ except that $\mathcal{P}$ is replaced by the smaller set $\mathcal{S}$ in (\ref{eta map}). 

For $i,j \in Q_0$ and $p \in e_jkQe_i$ (or, by abuse of notation, $p \in e_jAe_i$), denote by $\bar{\eta}(p)$ resp.\ $\bar{\tau}(p)$ the single nonzero matrix entry of $\eta(p)$ resp.\ $\tau(p)$. 
We will often write $\overbar{p}$ for $\bar{\eta}(p)$ or $\bar{\tau}(p)$.
Then, for $x \in \mathcal{P}$ and an arrow $a \in Q_1$, we have
\begin{equation*}
x \ni a \ \ \ \ \Longleftrightarrow \ \ \ \ x \mid \overbar{a}.
\end{equation*}

Ghor algebras $A = kQ/\operatorname{ker}\eta$ are quotients of dimer algebras $kQ/I$ since $I \subseteq \operatorname{ker}\eta$.
Indeed, if $pa$, $qa$ are unit cycles with $a \in Q_1$ an arrow, then the paths $p$ and $q$ are labeled by monomials $\overbar{p}$, $\overbar{q}$ that include precisely one variable for each perfect matching which does not contain $a$,
\begin{equation*}
\overbar{p} = \prod_{\substack{x \in \mathcal{P} :\\ x \not \ni a}} x = \overbar{q}.
\end{equation*}
Therefore $\eta(p) = e_{\operatorname{h}(p),\operatorname{t}(p)} \overbar{p} = \eta(q)$, so $p - q \in \ker \eta$.

We illustrate the difference between dimer algebras and ghor algebras in the following example. 
For a dimer quiver $Q$ on a surface, we denote by $Q^+$ the lift of $Q$ to the universal cover of the surface and call it the {\em covering quiver of $Q$}.

\begin{figure}
\centering \includegraphics[width=1.0\linewidth]{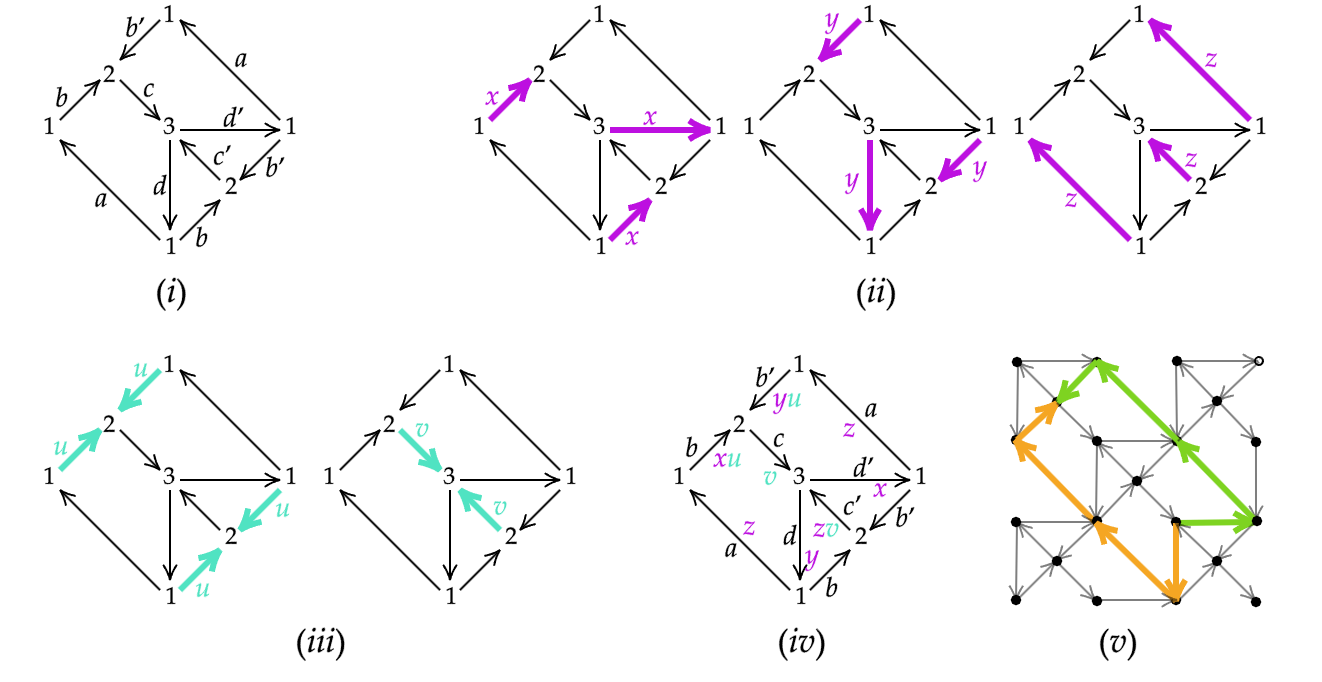}
\caption{The dimer quiver $Q$ and its perfect matchings for Example \ref{ex:first-example}.}
\label{matchingfigure}
\end{figure}

\begin{Example} \label{ex:first-example} 
Consider the dimer quiver $Q$ on a torus in Figure \ref{matchingfigure}.
A fundamental domain of $Q$ is shown in (i).
We claim that for this quiver, the ghor algebra $A = kQ/\ker \eta$ and dimer algebra $kQ/I$ of $Q$ are different.

To determine $A$, observe that there are five perfect matchings of $Q$: three which are simple, labeled $x,y,z$ in (ii), and two which are not, labeled $u,v$  in (iii).
The matching $u$ resp.\ $v$ is not simple because $Q \setminus u$ resp.\ $Q \setminus v$ has a source at vertex $2$ resp.\ $3$.
These matchings yield the monomial labeling of arrows given in (iv).
The ghor algebra $A$ is then obtained by identifying any two paths with coincident tails, coincident heads, and equal monomial labelings.
For example, the orange path $baad$ and green path $b'aad'$ in (v), drawn on the cover $Q^+$, have equal monomial labelings,
\begin{equation*}
    \overbar{baad} = (xu)zzy = (yu)zzx = \overbar{b'aad'},
\end{equation*}
and are therefore equal in $A$ since the variables 
$x,y,z,u,v$ commute.
These two paths, however, are clearly not equal in the dimer algebra $kQ/I$.

The algebra homomorphism $\eta$ in (\ref{eta map}) is explicitly given by 
$\eta(e_i) = e_{ii}$ and
\begin{equation*}
\begin{gathered}
\eta(a) = \left(\begin{smallmatrix} z & 0 & 0 \\ 0 & 0 & 0 \\ 0 & 0 & 0 \end{smallmatrix} \right), \ \ \ 
\eta(b) = \left( \begin{smallmatrix} 0 & 0 & 0 \\ xu & 0 & 0 
 \\ 0 & 0 & 0 \end{smallmatrix} \right), \ \ \ 
\eta(b') = \left( \begin{smallmatrix} 0 & 0 & 0 \\ yu & 0 & 0 
 \\ 0 & 0 & 0 \end{smallmatrix} \right),\\ 
 \eta(c) = \left( \begin{smallmatrix} 0 & 0 & 0 \\ 0 & 0 & 0 
 \\ 0 & v & 0 \end{smallmatrix} \right),
  \ \ \ 
\eta(c') = \left( \begin{smallmatrix} 0 & 0 & 0 \\ 0 & 0 & 0 
 \\ 0 & zv & 0 \end{smallmatrix} \right), \ \ \ 
 \eta(d) = \left(\begin{smallmatrix} 0 & 0 & y \\ 0 & 0 & 0 \\ 0 & 0 & 0 \end{smallmatrix} \right), \ \ \ 
\eta(d') = \left( \begin{smallmatrix} 0 & 0 & x \\ 0 & 0 & 0 
 \\ 0 & 0 & 0 \end{smallmatrix} \right).
 \end{gathered}
\end{equation*}
Note that the row resp.\ column of each nonzero matrix entry is specified by the head resp.\ tail of the corresponding arrow (with $Q_0 = \{ 1,2,3\}$). 
The algebra homomorphism $\tau$ is then obtained by setting $u = v = 1$ in $\eta$. 
\end{Example} 

Let $Q$ be a dimer quiver on $\Sigma$, and let $A = kQ/\ker \eta$ be its corresponding ghor algebra.
The first homology group $H_1(\Sigma) := H_1(\Sigma, \mathbb{Z})$ of $\Sigma$ plays a prominent role in the structure of $A$. 
Recall that $H_1(\Sigma)$ is the free abelian group on cycles modulo homology relations:\ two cycles $p, q$ are homologous if there is an oriented surface in $\Sigma$ whose boundary is $p - q$ (i.e. cutting the surface along $p$ and backwards along $q$ cuts the surface in two)\footnote{Note that the concatenation of cycles $p,q$ in $kQ$ is written multiplicatively, $pq$, whereas the concatenation of cycles in $H_1(\Sigma)$ is written additively, $p + q$.}. 
Figure~\ref{fig:octagon}.i shows an example of two cycles $p,q$ that are homologous but not homotopic.

Using $H_1(\Sigma)$, we call $A$ geodesic if there are enough `straight line paths' in $Q$:
 
\begin{Definition}\label{geodesic def} \cite[Definition 2.6]{BB21a}
\begin{itemize}
 \item[(i)] A cycle $p$ in $Q$ is \textit{geodesic} if there does not exist a cyclic permutation of a representative of $p+\ker\eta$ whose lift to $Q^+$ has a nontrivial cycle subpath.\\
 Two cycles in $Q$ are \textit{parallel} if they do not transversely intersect.
 \item[(ii)] 
We call $Q$ and $A$ \textit{geodesic} if for each cycle $p$ with no nontrivial cycle subpath modulo $\ker \eta$, there is a set $C_{[p]}$ of parallel geodesic cycles that covers $Q_0$ and contains a cycle $q \in C_{[p]}$ with $[q] = [p] \in H_1(\Sigma)$.
\end{itemize}
\end{Definition}

\begin{figure}
\centering \includegraphics[width=1.0\linewidth]{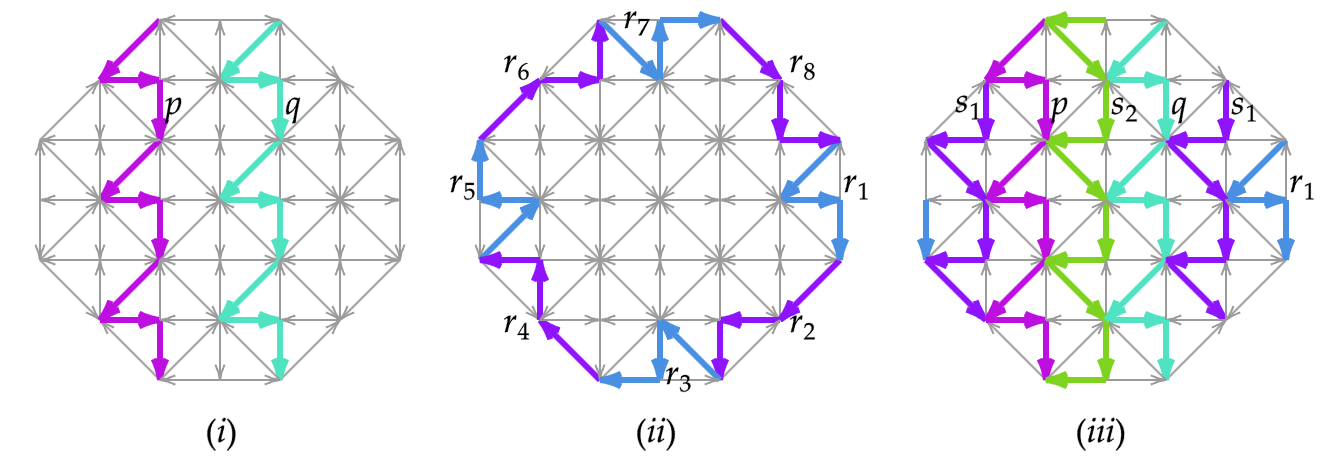}
\caption{ 
A geodesic ghor algebra on a genus $2$-surface $\Sigma$, obtained by identifying the opposite sides, and the corner vertices, of the octagon.
(i) The two cycles $p,q$ are geodesic, and homologous but not homotopic. 
(ii) A set of geodesic cycles $r_1, \ldots, r_8$ 
that spans $H_1(\Sigma)$ over $\mathbb{N}$. 
(iii) A set of parallel geodesic cycles $C_{[p]} = \{p,q, r_1, s_1,s_2\}$ that covers $Q_0$ and contains $p$.}
\label{fig:octagon}
\end{figure}

Observe that the orange and green cycles in Figure~\ref{matchingfigure}.v are not geodesic.
In fact, the ghor algebra itself in Example~\ref{ex:first-example} 
is not geodesic since there is no set $C_{[a]}$ of parallel geodesic cycles that covers $e_2,e_3$ and contains $a$.
In contrast, the ghor algebra on a genus $2$ surface in Figure~\ref{fig:octagon} is geodesic \cite[3.2]{BB21}. 
The two homologous cycles $p,q$ in (i) are both geodesic, and (iii) shows a set $C_{[p]}$ of parallel geodesic cycles that covers $Q_0$ and contains $p$. 
Similarly, the ghor algebra on a pinched torus in Figure~\ref{fig:simple-modules} below is geodesic \cite[3.3]{BB21}.

It was shown in \cite[Lemma 4.3]{BB21a} that noetherian ghor algebras on a torus (i.e., Calabi-Yau dimer algebras) are geodesic. If $A$ is geodesic, then $\ker \eta = \ker \tau$ \cite[Corollary 3.14]{BB21a}.
In this case, it suffices to consider the smaller set of simple matchings, rather than all perfect matchings, to determine the relations of $A$.

\begin{Notation} \label{not:sigma}
We denote by $\sigma_i$ a unit cycle at vertex $i \in Q_0$.
Observe that the monomial $\bar{\eta}(\sigma_i)$ resp.\ $\bar{\tau}(\sigma_i)$ is the product of all perfect resp.\ simple matchings, and thus does not depend on $i$ (see Example~\ref{ex:sigma-cycle} below). 
Consequently, all of the unit cycles at a fixed vertex are equal in $A$. 
To shorten notation, we write $\sigma$ for the monomial 
$\bar{\eta}(\sigma_i)$ or $\bar{\tau}(\sigma_i)$, depending on the context. 
\end{Notation}

\begin{Example}\label{ex:sigma-cycle}
For the ghor algebra given in Example \ref{ex:first-example},
we have $\sigma = xyzuv$ since, for example,
\begin{equation*}
\eta(\sigma_1) = \eta(d')\eta(c')\eta(b') 
= \left( \begin{smallmatrix} 0 & 0 & x \\ 0 & 0 & 0 
 \\ 0 & 0 & 0 \end{smallmatrix} \right)
 \left( \begin{smallmatrix} 0 & 0 & 0 \\ 0 & 0 & 0 
 \\ 0 & zv & 0 \end{smallmatrix} \right)
 \left( \begin{smallmatrix} 0 & 0 & 0 \\ yu & 0 & 0 
 \\ 0 & 0 & 0 \end{smallmatrix} \right)
= \left( \begin{smallmatrix} 
xyzuv & 0 & 0 \\ 0 & 0 & 0 \\ 0 & 0 & 0
\end{smallmatrix}
\right)
= \left( \begin{smallmatrix} 
\sigma & 0 & 0 \\ 0 & 0 & 0 \\ 0 & 0 & 0
\end{smallmatrix}
\right).
\end{equation*}
By cyclically permuting $\sigma_1 = d'c'b'$, we similarly have
\begin{equation*}
\eta(\sigma_2) = \eta(b')\eta(d')\eta(c') = \left( \begin{smallmatrix} 
0 & 0 & 0 \\ 0 & \sigma & 0 \\ 0 & 0 & 0
\end{smallmatrix}
\right), 
\ \ \ \ \
\eta(\sigma_3) = \eta(c')\eta(b')\eta(d') = \left( \begin{smallmatrix} 
0 & 0 & 0 \\ 0 & 0 & 0 \\ 0 & 0 & \sigma
\end{smallmatrix}
\right).
\end{equation*}
\end{Example}

Geodesic ghor algebras reflect the topology of the surface in which they are embedded:\ suppose $A$ is geodesic, and let $p,q$ be cycles in $Q$.
Then by \cite[Theorem 3.11]{BB21a},
\begin{equation} \label{p = q sigma}
[p] = [q] \ \ \ \ \Longleftrightarrow \ \ \ \ \overbar{p} = \overbar{q}\sigma^{\ell} \text{ for some } \ell \in \mathbb{Z}.
\end{equation}
In particular, if $\Sigma$ is smooth, then $p$ and $q$ are homologous if and only if $\overbar{p} = \overbar{q} \sigma^{\ell}$ for some $\ell \in \mathbb{Z}$.

The main tools we will use in our study of ghor algebras are the two subalgebras of the polynomial ring $k[\mathcal{S}]$,
\begin{equation} \label{cycle algebra def}
R := k[\cap_{i \in Q_0} \bar{\tau}(e_iAe_i)] \ \ \ \ \text{ and } \ \ \ \  
S := k[\cup_{i \in Q_0} \bar{\tau}(e_iAe_i)].
\end{equation}
By \cite[Corollary 3.15]{BB21a}, $R$ is isomorphic to the center of $A$.
We call the finitely generated overring $S$ of $R$ the \textit{cycle algebra} of $A$. 

We conclude with a lemma that motivates why we consider ghor algebras in place of dimer algebras on surfaces with nonzero curvature. Specifically, we are interested in relationships between the geometry of the center of the algebra and the topology of the surface, and the lemma shows that the center of a dimer algebra on a curved surface $\Sigma$ has no relation to any topological properties of $\Sigma$, such as its genus.

\begin{Lemma} \label{lm:no dimer}
Let $kQ/I$ be a dimer algebra on a surface $\Sigma$ of genus $g \geq 2$.
Then its center $Z(kQ/I)$ is the polynomial ring in one variable, $Z(kQ/I) \cong k[\sigma]$.
\end{Lemma}

\begin{proof}
First observe that the sum $\sigma$ of unit cycles $\sum_{j \in Q_0} \sigma_j$ is in $Z := Z(kQ/I)$ by the dimer relations $I$ in (\ref{I}).
Thus, $k[\sigma]$ is isomorphic to a subalgebra of $Z$.

Let $z \in Z \setminus k$ be another central element.
We may assume that $p = ze_i$ is a cycle since $ze_i = ze_i^2 = e_ize_i$, and $I$ is generated by differences of paths.

(i) Assume to the contrary that $[p] \not = 0$ in $H_1(\Sigma)$.
Since $g \geq 2$ and $[p] \not = 0$, there is a cycle $q \in e_i(kQ/I)e_i$ for which the cycles $pq$ and $qp$ are homologous but not homotopic. 
Thus, since each generator of $I$ is a homotopy relation between paths in $Q$, we have $pq \not = qp$ in $kQ/I$.
But then $p$ is not in $Ze_i$, a contradiction.

(ii) So assume to the contrary that $[p] = 0$, but $p \not = \sigma_i^n$ for any $n \geq 1$.
It suffices to suppose that, in the cover of $\Sigma$, no representative of $p$ lifts to a cycle $p^+$ that intersects itself.
Let $\mathcal{R}_p$ be the compact region in the cover bounded by a lift $p^+$ of a representative such that $\mathcal{R}_p$ contains a maximal number of unit cycles. 
By possibly cyclically permuting $p$, we may assume that there is an arrow $a \in Q_1$ whose lift $a^+$ has tail at $\operatorname{t}(p^+)$, lies outside of $\mathcal{R}_p$, and for which $ap$ contains no unit cycle subpath (modulo $I$).
Then, by the maximality of $\mathcal{R}_p$, $ap$ cannot be homotopic to a path $p'a$ using the dimer relations $I$.
Consequently $p$ cannot be in $Ze_i$, again a contradiction. 
\end{proof}

\section{Finite dimensional and simple representations of ghor algebras}

Let $A$ be a $k$-algebra and $V$ a finite dimensional left $A$-module.
Recall that such a module $V$ admits a composition series, that is, a filteration
\begin{equation*} \label{composition}
0 = V_0 \subset V_1 \subset V_2 \subset \cdots \subset V_{\ell} = V
\end{equation*}
with simple factors $V_j/V_{j-1}$.
We denote the representation of $A$ that defines its action on $V_j$ by $\rho_j: A \to \operatorname{End}_k(V_j)$.
Set $\rho := \rho_{\ell}$.

The projective dimension of $V$ is bounded above by the projective dimensions of the simple factors $V_j/V_{j-1}$:
\begin{equation} \label{bound}
\pd_A(V) \leq \max\{ \pd_A(V_j/V_{j-1}) \, | \, j \in [1, \ell] \}.
\end{equation}
Indeed, for each $j \in [1, \ell]$ there is a short exact sequence $0 \to V_{j-1} \to V_j \to V_j/V_{j-1} \to 0$, whence a long exact sequence
\begin{multline*}
0 \to \Hom(V_j/V_{j-1}, - ) \to \Hom(V_j, - ) \to \Hom(V_{j-1}, - ) \\
\to \Ext^1(V_j/V_{j-1}, - ) \to \Ext^1(V_j, - ) \to \Ext^1(V_{j-1}, - ) \to \cdots.
\end{multline*}
Thus, $\pd_A(V_j) \leq \max\{ \pd_A(V_{j-1}), \pd_A(V_j/V_{j-1}) \}$. 
The bound (\ref{bound}) then follows. 

\begin{Lemma} \label{central ann}
If $V$ is indecomposable, then for each $j \in [1, \ell]$ we have
\begin{equation*}
\ann_Z(V_j/V_{j-1}) = \ann_Z(V_1) \in \Max Z.
\end{equation*}
\end{Lemma}

\begin{proof}
Let $z \in Z$ be a nonzero central element of $A$; then $z$ is an endomorphism of $V$.
Since $V$ is indecomposable and of finite length, $\rho(z)$ is either a scalar multiple of the identity or nilpotent, by \cite[Theorem 19.17]{Lam01}.

If $\rho(z)$ is a scalar multiple of the identity, say $\rho(z) = \alpha$ for some $\alpha \in k$, then $z - \alpha$ is in $\ann_Z(V_j/V_{j-1})$ for each $j \in [1,\ell]$.

So suppose $\rho(z)$ is nilpotent. 
Since $V_j/V_{j-1}$ is simple, $z$ acts on $V_j/V_{j-1}$ by either a scalar multiple of the identity or by zero, by Schur's lemma.
But $\rho(z)$ is nilpotent, and so $z$ cannot act by a scalar multiple of the identity on $V_j/V_{j-1}$.
Thus, $z$ acts by zero on $V_j/V_{j-1}$.
Therefore $z$ is in $\ann_Z(V_j/V_{j-1})$.

Finally, we show that $\ann_Z(V_1)$ is a maximal ideal of $Z$.
The submodule $V_1 \subseteq V$ is simple, so the image $\rho_1(Z)$ is the base field $k$ by Schur's lemma.
Consequently, the kernel of the map $Z \stackrel{\rho_1}{\to} k$ is a maximal ideal of $Z$.
But this kernel is the annihilator $\ann_Z(V_1)$, and therefore $\ann_Z(V_1)$ is a maximal ideal of $Z$.
\end{proof}

Thus, in determining global dimension, it suffices to look at simple modules: 

\begin{Proposition} \label{main0}
Let $V$ be a finite dimensional indecomposable $A$-module.
Then
\begin{equation*}
\pd_A (V) \leq \max \{ \pd_A(U) \, | \, U \, \text{simple and $\ann_Z(U) = \ann_Z(V)$} \}. 
\end{equation*}
\end{Proposition}

\begin{proof}
We have
\begin{align*}
\pd_A (V) & \stackrel{(\textsc{i})}{\leq} \max \{ \pd_A(V_j/V_{j-1}) \, | \, j \in [1, \ell] \}\\
& \stackrel{(\textsc{ii})}{\leq} \max \{ \pd_A(U) \, | \, \text{$U$ simple and $\ann_Z(U) = \ann_Z(V)$} \}, 
\end{align*}
where (\textsc{i}) holds by (\ref{bound}), and (\textsc{ii}) follows from Lemma \ref{central ann}. 
\end{proof}

Now let $A = kQ/\ker \eta$ be a ghor algebra on a surface $\Sigma$, let $V$ be a left $A$-module, and let $0 \subset V_1 \subset \cdots \subset V_{\ell} = V$ be a composition series of $V$.
In the following we show that each simple factor $V_j/V_{j-1}$ has $k$-dimension at most $|Q_0|$. 

\begin{Lemma} \label{lm:simple}
Let $V$ be a simple $A$-module.
Then for each $i \in Q_0$, the $k$-dimension of $e_iV$ is at most $1$, $\operatorname{dim} e_iV \leq 1$.
\end{Lemma}

\begin{proof}
For each $i \in Q_0$, the subspace $e_iV$ of $V$ is a simple $e_iAe_i$-module.
The lemma then follows since $e_iAe_i$ is a commutative $k$-algebra and $k$ is an algebraically closed field. 
\end{proof}

\begin{Example}\label{ex:simple-modules}
Let $A$ be the ghor algebra on a pinched torus $\Sigma$ given in Figure~\ref{fig:simple-modules}. 
The figure shows the support of four families of isomorphism classes of simple $A$-modules, each parameterized by a subvariety of the algebraic variety $\operatorname{Max}S$.  
In each example, an arrow is represented by a nonzero scalar if and only if it is highlighted in blue.
In the first example, there is a one-dimensional vector space at every vertex. 
In the second and third examples, a vertex supports a nonzero vector space if and only if it is incident with a blue arrow.
In the last example, the only nonzero vector space is at the blue vertex. 
We will call such a one-dimensional module a \textit{vertex simple module}.
\end{Example}

\begin{figure}
\centering
\includegraphics[width=1.0\linewidth]{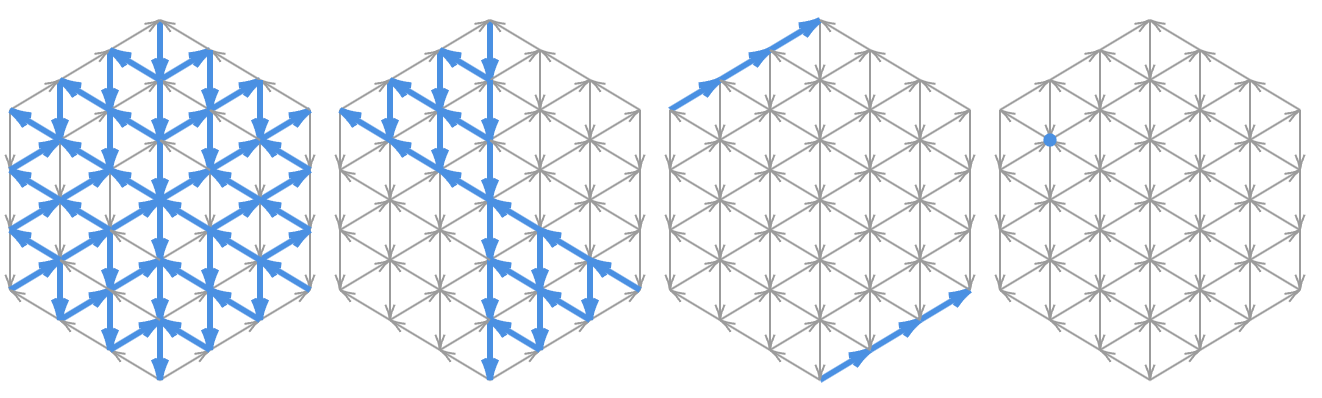}
\caption{The support of some simple modules, drawn in blue, over a ghor algebra on a pinched torus.}
\label{fig:simple-modules}
\end{figure}

\section{Cyclic localizations} \label{central localizations}

Let $A$ be a $k$-algebra with prime center $R$.
The central localization of $A$ at a prime ideal $\mathfrak{p} \in \operatorname{Spec}R$ is the tensor product 
\begin{equation} \label{central localization}
A_{\mathfrak{p}} := A \otimes_R R_{\mathfrak{p}}.
\end{equation}
The algebras we consider here are (isomorphic to) matrix rings $A = [A_{ij}]_{i,j} \subseteq M_n(B)$, where $B$ is an integral domain $k$-algebra that contains $R$.\footnote{Here we are taking $R$ to be a subalgebra of $B$, and the center of $A$ to be $R \, 1_A = R \, \bm{1}_{n}$, which we simply identify with $R$ itself.}
Note that each $A_{ij} \subseteq B$ is a $k$-vector space, and each $A_i := A_{ii}$ is a $k$-algebra.
The central localization (\ref{central localization}) may be written
\begin{equation*}
A_{\mathfrak{p}} = \left\langle [A_{ij}R_{\mathfrak{p}}]_{i,j} \right\rangle \subseteq M_n(\operatorname{Frac}B).
\end{equation*}

In the nonnoetherian setting, so-called cyclic localizations  are more amiable than central localizations. 
For example, cyclic localizations of ghor algebras remain path algebras modulo relations, in contrast to their central localizations; this is shown in Lemma \ref{lm:simple non max dim} below.
Cyclic localizations were introduced in \cite[Definition 3.1]{B18}, and are used there and in \cite{B21b} to study homological properties of certain nonnoetherian matrix rings.

\begin{Definition}
Let $B$ be an integral domain $k$-algebra, and let $A = [A_{ij}]_{i,j} \subseteq M_n(B)$ be a matrix ring.
The \textit{cycle algebra} of $A$ is the union $S = k[\cup_i A_i]$, and the \textit{cyclic localization} of $A$ at a prime $\mathfrak{q} \in \operatorname{Spec}S$ is the algebra
\begin{equation*}
A_{\mathfrak{q}} = \left\langle [A_{ij}(A_j)_{\mathfrak{q} \cap A_j} ]_{i,j} \right\rangle \subseteq M_n(\operatorname{Frac}B).
\end{equation*}
\end{Definition}

\begin{Remark} 
Observe that the cyclic and central localizations coincide whenever the center $R$ is given by the intersection $R = k[\cap_i A_i]$, and $R = S$.
In particular, if $R = S$, then $A_i = A_j$ for each $i,j \in [1,n]$, and so $A_i = R$ for each $i \in [1,n]$. 
Furthermore, under mild assumptions, $R = S$ if and only if $R$ is noetherian, if and only if $A$ is a finitely generated $R$-module \cite[Theorem 4.1.2]{B16}. 
\end{Remark}

Now let $A$ be a ghor algebra with cycle algebra $S$, and fix a prime ideal $\mathfrak{q} \in \operatorname{Spec}S$. 

\begin{Definition}
We call a path $p$ \textit{locally invertible} in $A_{\mathfrak{q}}$ if there is an element $q \in e_{\operatorname{t}(p)}A_{\mathfrak{q}} e_{\operatorname{h}(p)}$ such that
\begin{equation*}
qp = e_{\operatorname{t}(p)} \ \ \ \ \text{ and } \ \ \ \ pq = e_{\operatorname{h}(p)}.
\end{equation*}
\end{Definition}

\begin{Lemma} \label{lm:simple non max dim} \
Every subpath of a locally invertible cycle in $A_{\mathfrak{q}}$ is locally invertible.
\end{Lemma}

\begin{proof}
Let $p$ be a locally invertible cycle and let $a$ be a subpath of $p$.
By possibly cyclically permuting $p$, it suffices to suppose that $p$ factors into paths $p = aq$.
Set 
\begin{equation*}
a^{-1} := q \otimes \overbar{p}^{-1}.
\end{equation*}
Then $a^{-1}$ is a path from $e_{\operatorname{t}(q)} = e_{\operatorname{h}(a)}$ to $e_{\operatorname{h}(q)} = e_{\operatorname{t}(a)}$, and $\overbar{a^{-1}} = \overbar{q} \otimes \overbar{p}^{-1} = \overbar{a}^{-1}$.
\end{proof}

We call an element of a localization $A_{\mathfrak{q}}$ a \textit{path} if it is a concatenation of arrows and inverse arrows in $A_{\mathfrak{q}}$.
If $a \in Q_1$ is an arrow and $c \in Q_{\geq 1}$ is path for which $ac$ is a unit cycle, then we call $c$ the \textit{complementary arc} to $a$.

\begin{Lemma}
Let $p,q$ be paths in a cyclic localization $A_{\mathfrak{q}}$.
If
\begin{equation} \label{coincident}
\operatorname{t}(p^+) = \operatorname{t}(q^+) \ \ \ \ \text{ and } \ \ \ \ \operatorname{h}(p^+) = \operatorname{h}(q^+),
\end{equation}
then $\overbar{p} = \overbar{q} \sigma^n$ for some $n \in \mathbb{Z}$.
\end{Lemma}

\begin{proof}
Fix a  locally invertible arrow $a \in Q_1$, and denote by $c$ its complementary arc. 
Then $a^{-1}\sigma_{\operatorname{t}(c)} = c$.
Thus, if $q$ is a path in $A_{\mathfrak{q}}$ consisting of precisely $\ell \geq 0$ inverse arrows $a^{-1}$, then $q\sigma^{\ell}_{\operatorname{t}(q)}$ is a path in $A$.
Furthermore, if $p$ is a path in $A$ for which (\ref{coincident}) holds, then
\begin{equation*}
\operatorname{t}((q\sigma^{\ell}_{\operatorname{t}(q)})^+) = \operatorname{t}(q^+) = \operatorname{t}(p^+) \ \ \ \ \text{ and } \ \ \ \ \operatorname{h}((q\sigma^{\ell}_{\operatorname{t}(q)})^+) = \operatorname{h}(q^+) = \operatorname{h}(p^+).
\end{equation*}
Therefore there is an $n \in \mathbb{Z}$ such that $\overbar{q}\sigma^{\ell + n} = \overbar{p}$, by (\ref{p = q sigma}).
\end{proof}

\section{The first syzygy}\label{sec:syzygy}

Throughout, let $A$ be a geodesic ghor algebra and $V$ a simple left $A$-module. 
By \textit{cycle} we mean a nontrivial cycle, and by \textit{cyclic subpath} we mean a proper subpath that is a cycle.

For each path $p \in Q_{\geq 0}$, denote by $\tilde{p} \in k$ the scalar representing $p$ with $V$ viewed as a representation of $Q$, i.e. $\tilde{p}$ is the scalar by which $p$ acts.
Set 
\begin{equation*}
\mathfrak{n} := \operatorname{ann}_S V := \{g \in S \ | \text{ if $s \in e_iAe_i$ satisfies $\bar{\tau}(s) = g$, then $sV = 0$} \}.
\end{equation*}

\begin{Lemma}
The annihilator $\mathfrak{n}$ is a maximal ideal of $S$.
\end{Lemma}

\begin{proof}
The cycle algebra $S$ is finitely generated over $k$ since there are only a finite number of cycles in $Q$ without cyclic subpaths. 
Thus, $S$ has a finite generating set consisting of monomials; let $g$ be a monomial in this generating set. 
Let $i,j \in Q_0$ be vertices for which there are cycles $s_i \in e_iAe_i$, $s_j \in e_jAe_j$ such that $\bar{\tau}(s_i) = \bar{\tau}(s_j) = g$.

Suppose $\tilde{e}_i \not = 0$ and set $c := \tilde{s}_i \in k$.
Then $s_i - ce_i$ annihilates $V$, and 
\begin{equation*}
\bar{\tau}(s_i - ce_i) = g - c.
\end{equation*}
If $\tilde{e}_j \not = 0$ as well, then $\tilde{s}_j = c$ since $V$ is simple, by Schur's lemma.
If instead $\tilde{e}_j = 0$, then both $s_j$ and $e_j$ annihilate $V$.
In either case, $s_j - ce_j$ also annihilates $V$, and 
\begin{equation*}
\bar{\tau}(s_j - ce_j) = g -c = \bar{\tau}(s_i - ce_i).
\end{equation*}
Therefore, $g-c$ is in $\mathfrak{n}$. 
But the field $k$ is algebraically closed, and so $\mathfrak{n}$ is a maximal ideal of $S$ by Hilbert's Nullstellensatz. 
\end{proof}

Recall the noetherian locus $U_{S/R}$ defined in (\ref{noetherian locus}).

\begin{Proposition} \label{max dim prop}
Let $V$ be a simple $A$-module of dimension vector $\underline{\dim} V = 1^{Q_0}$.
Set $\mathfrak{n} := \operatorname{ann}_S V \in \operatorname{Max}S$ and $\mathfrak{m} := \mathfrak{n} \cap R \in \operatorname{Max}R$. 
Then $R_{\mathfrak{m}} = S_{\mathfrak{n}}$, and the corresponding central and cyclic localizations coincide, 
$A_{\mathfrak{m}} = A_{\mathfrak{n}}$. 
In particular, $\mathfrak{n}$ is in the noetherian locus $U_{S/R} \subset \operatorname{Max}S$ of $R$.
\end{Proposition}

\begin{proof}
Fix $i \in Q_0$.
Since $V$ is simple and $\underline{\dim} V = 1^{Q_0}$, there is a cycle $t \in e_iAe_i$ which passes through each of $Q$ such that $\tilde{t} \not = 0$. 
Since $t$ passes through each vertex,  $\overline{t}$ is in $R$. 
Thus, since $\tilde{t} \not = 0$, $\overline{t}^{-1}$ is in $R_{\mathfrak{m}}$. 
Whence, $t^{-1}$ is in $A_{\mathfrak{m}} = A \otimes_R R_{\mathfrak{m}}$. 

Let $s \in A$ be any cycle. 
Factor $t$ into paths $t = t_2e_{\operatorname{t}(s)}t_1$. 
Then $\overbar{s} = \overline{t}^{-1}\overbar{t_2st_1}$ is in $R_{\mathfrak{m}}$.
Thus $S \subset R_{\mathfrak{m}} \subseteq S_{\mathfrak{n}}$, and so $S_{\mathfrak{n}} = R_{\mathfrak{m}}$. 
Consequently, 
\begin{equation*}
\bar{\tau}(e_i A_{\mathfrak{n}}e_i) \subseteq \cup_{j \in Q_0} \bar{\tau}(e_jA_{\mathfrak{n}}e_j) = S_{\mathfrak{n}} = R_{\mathfrak{m}} \subseteq \bar{\tau}(e_iA_{\mathfrak{m}}e_i) \subseteq \bar{\tau}(e_i A_{\mathfrak{n}}e_i).
\end{equation*} 
But then $\bar{\tau}(e_iA_{\mathfrak{n}}e_i) = \bar{\tau}(e_iA_{\mathfrak{m}}e_i)$.
Therefore $A_{\mathfrak{n}} = A_{\mathfrak{m}}$. 
\end{proof}

Let $\{ \varepsilon_i \ | \ i \in Q_0, \ \tilde{e}_i \not = 0 \}$ be a basis for $V$, where $e_i \varepsilon_j = \delta_{ij}\varepsilon_j$ whenever $\tilde{e}_i\not = 0$.
Consider the module epimorphisms over $A$ and over the cyclic localization $A_{\mathfrak{n}}$,
\begin{equation*} \label{delta map}
\delta: A \to V, \ \ \ \ \ \ \delta_{\mathfrak{n}}: A_{\mathfrak{n}} \to V,
\end{equation*}
defined on paths by $p \mapsto p \varepsilon_{\operatorname{t}(p)} = \tilde{p} \varepsilon_{\operatorname{h}(p)}$, and extended $k$-linearly to $A$ and $A_{\mathfrak{n}}$ respectively.

The kernels of the maps $\delta$ and $\delta_{\mathfrak{n}}$ are the first syzygies of projective resolutions of $V$ over $A$ and $A_{\mathfrak{n}}$. 
Since the relations of $A$ are binomial on the paths of $Q$, there are two types of generators for the kernels: 
\begin{itemize}
 \item[(a)] binomial generators, which are elements of the form $\tilde{p}q - \tilde{q}p$ for paths $p,q \in e_jA_{\mathfrak{n}}e_i$ satisfying $\tilde{p} \not = 0 \not = \tilde{q}$; and
 \item[(b)] monomial generators, which are paths $p$ satisfying $\tilde{p} = 0$. 
\end{itemize}

\begin{Proposition} \label{first}
As left ideals of $A_{\mathfrak{n}}$,
\begin{multline*} \label{kf2}
_{A_{\mathfrak{n}}} \! \left\langle \tilde{p}q - \tilde{q}p \ | \ \text{$p,q \in A_{\mathfrak{n}}$ paths with $\operatorname{t}(p) = \operatorname{t}(q)$, $\operatorname{h}(p) = \operatorname{h}(q)$, and $\tilde{p} \not = 0 \not = \tilde{q}$} \right\rangle \\ 
= \,
_{A_{\mathfrak{n}}} \! \left\langle s - \tilde{s}e_{\operatorname{t}(s)} \ | \ \text{$s \in A$ a cycle with no cyclic subpaths and $\tilde{s} \not = 0$} \right\rangle. 
\end{multline*}
\end{Proposition}

\begin{proof}
Let $p,q$ be paths in $A_{\mathfrak{n}}$ with $\operatorname{t}(p) = \operatorname{t}(q)$, $\operatorname{h}(p) = \operatorname{h}(q)$, and $\tilde{p} \not = 0 \not = \tilde{q}$.
It suffices to omit the scalars $\tilde{p}, \tilde{q}, \tilde{s} \in k$ and show that $p - q$ is in the left ideal
\begin{equation*}
J := \, _{A_{\mathfrak{n}}} \! \left\langle s - e_{\operatorname{t}(s)} \ | \ \text{$s \in A$ a cycle with no cyclic subpaths and $\tilde{s} \not = 0$} \right\rangle.
\end{equation*}

In $A_{\mathfrak{n}}$, each arrow supported on $V$ is locally invertible, by Lemma \ref{lm:simple non max dim}.
Fix $i \in Q_0$.
Let $s,t$ be cycles in $e_iAe_i$ with $\tilde{s} \not = 0 \not = \tilde{t}$.
Then
\begin{equation*}
\begin{split}
ts - e_i & = t(s-e_i) + (t - e_i) \in J, \\
s^{-1} - e_i & = -s^{-1}(s - e_i) \in J, \\
s^{-2} - e_i & = -s^{-1}(s^{-1} + e_i)(s - e_i) \in J. 
\end{split}
\end{equation*}
It thus follows by induction that
\begin{equation} \label{kh}
_{A_{\mathfrak{n}}} \! \left\langle s^n - e_{\operatorname{t}(s)} \ | \ \text{$s \in A$ a cycle with $\tilde{s} \not = 0$; $n \in \mathbb{Z}$} \right\rangle \subseteq J.
\end{equation}

Now suppose $s,t$ factor into paths $s = s_2s_1$, $t = t_2t_1$ in $A$, with $\operatorname{h}(s_1) = \operatorname{h}(t_1)$.
Then (\ref{kh}) implies
\begin{equation} \label{kh2}
t_1^{-1}s_1 - e_i = (s_2t_1)^{-1}(s - e_i) + ((s_2t_1)^{-1} - e_i) \in J.
\end{equation}
Whence,
\begin{equation} \label{kf}
s_1 - t_1 = t_1(t_1^{-1}s_1 - e_i) \in J. 
\end{equation}
Furthermore, if $t_1$ factors into paths $t_1 = r_2r_1$ in $A$, and $s_3 \in e_iAe_{\operatorname{h}(r_1)}$ is another path supported on $V$, then (\ref{kh2}) implies
\begin{equation*}
s_3r_2^{-1}s_1 - e_i = s_3r_1(t_1^{-1}s_1 - e_i) + (s_3r_1 - e_i) \in J.
\end{equation*}
Therefore, by induction, if $r$ is any cycle \textit{in the localization} $A_{\mathfrak{n}}$ with $\tilde{r} \not = 0$, then $r - e_{\operatorname{t}(r)} \in J$. 
Consequently, (\ref{kf}) implies that for our original paths $p,q \in A_{\mathfrak{n}}$, their difference $p - q$ is in $J$.
\end{proof}

\begin{Lemma} \label{same}
If $i,j \in Q_0$ are vertices for which $\tilde{e}_i \not = 0 \not = \tilde{e}_j$, then 
\begin{equation*}
\bar{\tau}(e_iA_{\mathfrak{n}}e_i) = \bar{\tau}(e_jA_{\mathfrak{n}}e_j).
\end{equation*}
\end{Lemma}

\begin{proof}
Let $s$ be a cycle in $e_jA_{\mathfrak{n}}e_j$.
Since $V$ is simple, there are paths $t_1 \in e_jAe_i$ and $t_2 \in e_iAe_j$ such that $\tilde{t}_1 \not = 0 \not = \tilde{t}_2$.
Set $t := t_2t_1$ (similar to the proof of Proposition \ref{max dim prop}).
Then the cycle $t^{-1} t_2st_1$ is in $e_iA_{\mathfrak{n}}e_i$ and has $\bar{\tau}$-image $\overbar{s}$.
\end{proof}

\begin{Proposition} \label{second}
As left ideals of $A_{\mathfrak{n}}$,
\begin{multline*}
_{A_{\mathfrak{n}}} \! \left\langle p \ | \ \text{$p \in A_{\mathfrak{n}}$ a path with $\tilde{p} = 0$ and $\tilde{e}_{\operatorname{h}(p)} \not = 0$} \right\rangle \\ \subseteq \,
_{A_{\mathfrak{n}}} \! \left\langle s \ | \ \text{$s \in A$ a cycle with no cyclic subpaths and $\tilde{s} = 0$} \right\rangle.
\end{multline*}
\end{Proposition}

\begin{proof}
Consider a path $p = aq_1 \in A$ with $\tilde{a} = 0 \not = \tilde{q}_1$ and $\tilde{e}_{\operatorname{h}(p)} \not = 0$.
Set $i := \operatorname{t}(p)$.

Since $V$ is simple and $\tilde{e}_{\operatorname{h}(p)} \not = 0$, there is a path $q_2 \in e_iAe_{\operatorname{h}(a)}$ such that $\tilde{q}_2 \not = 0$.
Set $s := q_2aq_1 \in e_iAe_i$.
We may assume that $s$ has no cyclic subpaths by Lemma \ref{same}.
Furthermore, $q_2$ is a locally invertible path in $A_{\mathfrak{n}}$, by Lemma \ref{lm:simple non max dim}.
Therefore $p = q^{-1}s$ is in $A_{\mathfrak{n}}s$.
\end{proof}

Denote by $\mathscr{P}_i$ the set of paths $p$ in $A_{\mathfrak{n}}e_i$ with no cyclic subpaths such that $\tilde{e}_{\operatorname{h}(p)} = 0$, and where $p$ is minimal with respect to this property.

\begin{Proposition} \label{third}
Fix $i \in Q_0$.
As left ideals of $A_{\mathfrak{n}}$,
\begin{multline*}
_{A_{\mathfrak{n}}} \! \left\langle p \ | \ \text{$p \in A_{\mathfrak{n}}e_i$ a path with $\tilde{e}_{\operatorname{h}(p)} = 0$} \right\rangle \\ \subseteq \,
_{A_{\mathfrak{n}}} \! \left\langle p, \, s - \tilde{s}e_i \ | \ \text{$p \in \mathscr{P}_i$; $s \in e_iAe_i$ a cycle with no cyclic subpaths} \right\rangle.
\end{multline*}
\end{Proposition}

\begin{proof}
Let $p \in A_{\mathfrak{n}}e_i$ be a path for which $\tilde{e}_{\operatorname{h}(p)} = 0$.
Let $s$ be a cycle in $e_iA_{\mathfrak{n}}e_i$.
We claim that $ps^n$ is in the left ideal $\left\langle p, s - \tilde{s}e_i \right\rangle \subset A_{\mathfrak{n}}$.
First observe that
\begin{equation*}
ps = p(s - \tilde{s}e_i) + \tilde{s}p \in \, _{A_{\mathfrak{n}}} \! \left\langle p, s - \tilde{s}e_i \right\rangle.
\end{equation*}
Then, by induction,
\begin{equation*}
ps^n = p(s^{n-1} - \tilde{s}^{n-1}e_i)(s - \tilde{s}e_i) + \tilde{s}(ps^{n-1}) + \tilde{s}^{n-1}(ps) - \tilde{s}^np \in \, _{A_{\mathfrak{n}}} \! \left\langle p, s - \tilde{s}e_i \right\rangle.
\end{equation*}
The proposition then follows from Lemma \ref{same}.
\end{proof}

\begin{Theorem} \label{main 2}
As left ideals of $A_{\mathfrak{n}}$,
\begin{equation*}
\begin{split}
\ker \delta_{\mathfrak{n}}
& = \,
_{A_{\mathfrak{n}}} \! \left\langle \mathscr{P}_i, \, s - \tilde{s}e_{\operatorname{t}(s)} \ | \ \text{$i \in Q_0$; $s \in A$ a cycle with no cyclic subpaths} \right\rangle. 
\end{split}
\end{equation*}
\end{Theorem}

\begin{proof}
Follows from Propositions \ref{first}, \ref{second}, and \ref{third}.
\end{proof}

\section{Projective resolutions of the simple modules}

Alongside cluster algebras, one of the main driving forces for the wide mathematical interest in quivers with potential, as well as the construction of Calabi-Yau algebras, was a specific projective resolution introduced by two physicists, Berenstein and Douglas \cite[Section 5.5]{BD02}.\footnote{In fact, the notion of a quiver with potential was introduced by Berenstein in \cite{BJL00,BL01}.}
We first recall this resolution for the special case of (vertex simple modules over) dimer algebras on a torus.
We will then generalize the resolution to the setting of geodesic ghor algebras on higher genus surfaces using a Koszul-like complex.

Let $A = kQ/I$ be a dimer algebra, where $I$ is the ideal of relations (\ref{I}).
We keep the notation from Section~\ref{sec:syzygy} and write $\tilde{p} \in k$ for the scalar representing a path $p$ in a given simple $A$-module.
Let $V$ be the vertex simple $A$-module at $i$, that is, $\operatorname{dim}_kV = 1$ with $\tilde{e}_j = \delta_{ij}$. 
The \textit{Berenstein-Douglas complex} of $V$ is the projective complex
\begin{multline*} \label{BD resolution}
0 \to Ae_i \stackrel{\cdot \left[ \begin{smallmatrix} b_1 & \cdots & b_n \end{smallmatrix} \right]}{\longrightarrow} \bigoplus_{b_k \in e_iQ_1} Ae_{\operatorname{t}(b_k)} \\
\stackrel{\cdot \left[ \begin{smallmatrix} u_1 & -v_2 & 0 & \cdots & 0 \\ 0 & u_2 & -v_3 & & 0 \\ 0 & 0 & u_3 & & 0 \\ \vdots & & & \ddots & \vdots \\ -v_1 & 0 & 0 & \cdots & u_{\ell} \end{smallmatrix} \right]}{\longrightarrow} \bigoplus_{a_j \in Q_1e_i} Ae_{\operatorname{h}(a_j)} \stackrel{\cdot \left[ \begin{smallmatrix} a_1 \\ \vdots \\ a_n \end{smallmatrix} \right]}{\longrightarrow} Ae_i \stackrel{\delta}{\longrightarrow} V \to 0,
\end{multline*}
where $b_j u_j a_j$, $b_{j-1} v_ja_j$ are counterclockwise and clockwise unit cycles at $i$ respectively, as illustrated in Figure \ref{BD figure}, and $\delta(p) = p \cdot 1 \in k \cong V$ \cite[5.12]{BD02}.

\begin{figure} 
\begin{equation*}
\begin{array}{ccc}
\includegraphics[scale=.23]{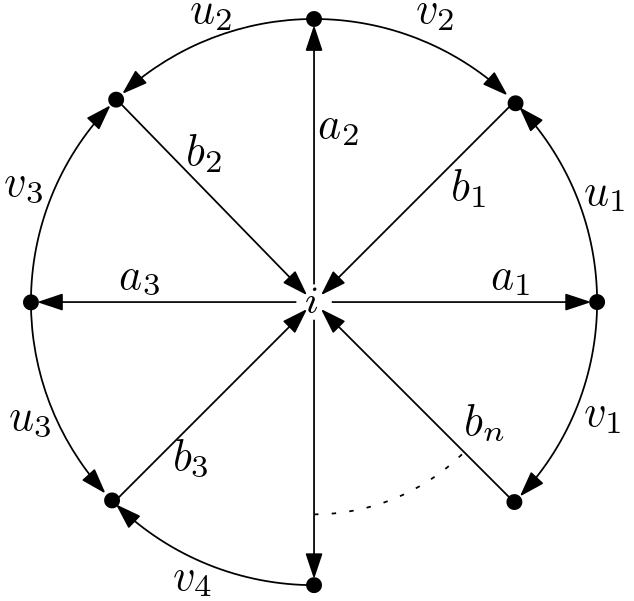} & \ \ \ \ &
\includegraphics[scale=.33]{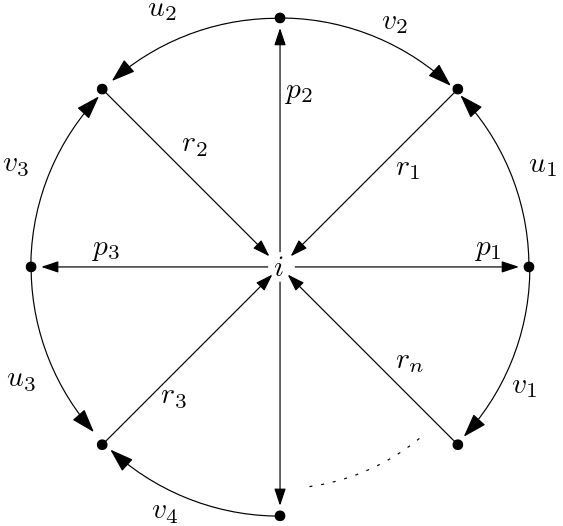}\\
(i) & & (ii)
\end{array}
\end{equation*}
\caption{The setup for the Berenstein-Douglas complex and its generalization in (\ref{homotopic BD}). Here, $a_j, b_j \in Q_1$ are arrows; $u_j,v_j \in Q_{\geq 0}$ are paths; and  $b_j u_j a_j$, $b_{j-1} v_ja_j$ are unit cycles.}
\label{BD figure}
\end{figure}

On a torus, commutation relations of cycles in the quiver $Q$ are homotopy relations, and so are accounted for in the Berenstein-Douglas complex.
Indeed, for cycles $p,q \in e_iAe_i$, we have
\begin{equation*}
pq = qp \ \ \ \ \Longleftrightarrow \ \ \ \ \text{$pq$ is homotopic to $qp$}.
\end{equation*}
This is illustrated in Figure \ref{torus figure}.
On genus $g \geq 2$ surfaces, however, commutation relations of cycles are not homotopy relations in general, since such a surface has nonzero curvature. 

\begin{figure}
\centering \includegraphics[width=1.0\linewidth]{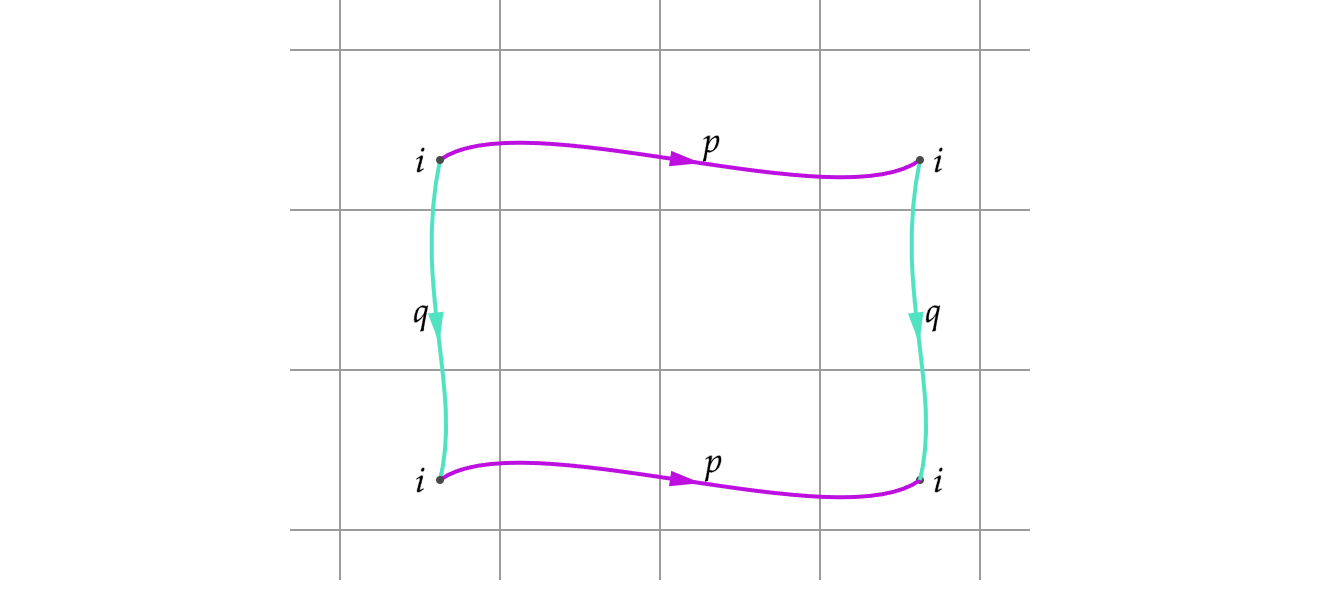}
\caption{On a torus, commutation relations of cycles are homotopy relations. The cycles $p,q \in e_iAe_i$ are drawn on the torus cover $\mathbb{R}^2$, with each square a fundamental domain.}
\label{torus figure}
\end{figure}

We now construct a general projective resolution of simple modules over localized geodesic ghor algebras.
Let $\mathfrak{n}$ be a maximal ideal of $S$, and fix a simple $A_{\mathfrak{n}}$-module $V$.
In particular, $\operatorname{ann}_S V = \mathfrak{n}$.
Recall the set of paths $\mathscr{P}_i$ from Proposition \ref{third}.
Set
\begin{equation*}
Q'_0 := \{ i \in Q_0 \, | \, \tilde{e}_i \not = 0 \}.
\end{equation*}
To construct a projective resolution of $V$ we proceed in two steps:
\\ \\
\indent (i) \textit{A modified Koszul part for non-homotopic homologous cycles.}
Recall that the fundamental polygon of $\Sigma$ has $2N$ sides, and $\sigma := \prod_{x \in \mathcal{S}}x$ is the product of all simple matchings as defined in Notation~\ref{not:sigma}.
Set 
\begin{equation*}
\begin{split}
T & := \{ (a_1, \ldots, a_N) \, | \, a_j \in \{-1,0,1\} \} \subset \mathbb{Z}^N \cong H_1(\Sigma, \mathbb{Z}),\\
T^{\sigma} & := T \oplus \{ \sigma \} \subset \mathbb{Z}^{N+1} \cong H_1(\Sigma, \mathbb{Z}) \oplus \mathbb{Z} \sigma.
\end{split}
\end{equation*}
Furthermore, for $m \in [1, N]$, set
\begin{equation*} \label{T_m}
\begin{split}
T_m & := \{ (\alpha_1, \ldots, \alpha_m) \, | \, \text{$\alpha_1, \ldots, \alpha_m \in T$ linearly indep.\ over $\mathbb{Z}$} \},\\
T^{\sigma}_m & := \{ (\alpha_1. \ldots, \alpha_m) \, | \, \text{$\alpha_1, \ldots, \alpha_m \in T^{\sigma}$ linearly indep.\ over $\mathbb{Z}$} \}.
\end{split}
\end{equation*}
Denote by $kT_m$ and $kT^{\sigma}_m$ the $k$-vector spaces with bases $T_m$ and $T^{\sigma}_m$.
For $i \in Q'_0$ and $m \in [1,N]$, consider the finite rank projective left $A_{\mathfrak{n}}$-modules
\begin{equation} \label{Pm}
P_{im} := \left\{ \begin{array}{ll}
A_{\mathfrak{n}} e_i \otimes_k kT^{\sigma}_m & \text{ if $\underline{\dim}V = 1^{Q_0}$}\\ 
A_{\mathfrak{n}} e_i \otimes_k kT_m & \text{ otherwise}
\end{array} \right., \ \ \ \ \ \ \ P_{i0} = A_{\mathfrak{n}}e_i,
\end{equation}
and connecting maps $\delta_{im}: P_{im} \to P_{i,m-1}$ defined by
\begin{equation} \label{cycle part}
\delta_{im}(\beta \otimes (\alpha_1, \ldots, \alpha_m)) = \sum_{j = 1}^m (-1)^{j+1} \beta (s_j - \tilde{s}_je_i) \otimes (\alpha_1, \ldots, \widehat{\alpha}_j, \ldots, \alpha_m),
\end{equation}
where $s_j$ is a geodesic cycle in $e_iA_{\mathfrak{n}}e_i$ with homology class $[s_j] = \alpha_j$.

For each $(\alpha_1, \alpha_2) \in T_2$, let $s$ be a geodesic cycle for which there are cycles $t_1, t_2$, with $[t_1] = \alpha_1$, $[t_2] = \alpha_2$, such that $s$ is homotopic to $t_1t_2$.
If $s$ has a trivial subpath $e_j$ for which $\tilde{e}_j = 0$, factor $s$ into subpaths $s = q_{12}p_{12}$ with $p_{12} \in \mathscr{P}_i$; in this case we say $s$ \textit{disappears}.

Let $k\mathscr{P}_i$ be the $k$-vector space with basis $\mathscr{P}_i$.
Consider the complex
\begin{equation*}
P_{i2} \stackrel{\delta'_{i2}}{\longrightarrow} P_{i1} \oplus \bigoplus_{p \in \mathscr{P}_i}
A_{\mathfrak{n}}e_{\operatorname{h}(p)} \otimes_k p \stackrel{\delta'_{i1}}{\longrightarrow} P_{i0},
\end{equation*}
where 
\begin{equation*}
\delta'_{i2}( \beta \otimes (\alpha_1, \alpha_2)) := \left\{ \begin{array}{ll}
\beta (q_{12} \otimes p_{12} - q_{21} \otimes p_{21}) & \text{ if $s$ disappears}\\
\delta_{i2}(\beta \otimes (\alpha_1, \alpha_2)) & \text{ otherwise}
\end{array} \right.
\end{equation*}
and for $\beta \otimes \alpha \in P_{i1}$, $p \in \mathscr{P}_i$,
\begin{equation*}
\delta'_{i1}(\beta \otimes \alpha) := \delta_1(\beta \otimes \alpha), \ \ \ \ \ \ 
\delta'_{i1}(\beta \otimes p) := \beta p.
\end{equation*}

(ii) \textit{A generalized Berenstein-Douglas part for homotopic paths.}
For $i \in Q'_0$, order the paths in $\mathscr{P}_i$ counterclockwise around $i$, say $p_1, \ldots, p_{\ell} \in \mathscr{P}_i$. 
For $j \in [1,\ell]$, let $u_j,v_j,r_j$ be minimal paths as shown in Figure \ref{torus figure}.ii satisfying the homotopy relations
\begin{equation} \label{homotopic BD}
u_jp_j - v_{j+1}p_{j+1} = 0, \ \ \ \ \ \  r_{j-1}v_j - r_ju_j = 0.
\end{equation}
From these relations we can construct a generalization of the Berenstein-Douglas complex:
\begin{equation*}
0 \longrightarrow A_{\mathfrak{n}}e_i \stackrel{\delta'_{i3}}{\longrightarrow} \bigoplus_{j = 1}^{\ell} A_{\mathfrak{n}}e_{\operatorname{t}(r_j)} \stackrel{\delta'_{i2}}{\longrightarrow} \bigoplus_{p \in \mathscr{P}_i} A_{\mathfrak{n}}e_{\operatorname{h}(p)} \otimes_k p,
\end{equation*}
where
\begin{equation*}
\begin{split}
\delta'_{i3}(\beta) & = \beta(r_1, \ldots, r_{\ell}),\\
\delta'_{i2}(\beta e_{\operatorname{t}(r_j)}) & = \beta( u_j \otimes p_j - v_{j+1} \otimes p_{j+1}).
\end{split}
\end{equation*}

\begin{Remark} \rm{
If $V$ has dimension vector $\underline{\dim}V = 1^{Q_0}$ (that is, maximal $k$-dimension by Lemma \ref{lm:simple}), then $\mathscr{P}_i = \varnothing$.
In this case, then, the Berenstein-Douglas complex will not be part of the projective resolution of $V$.
}\end{Remark}

Putting (i) and (ii) together yields a projective complex of $e_iV$,
\begin{multline} \label{proj resolution}
0 \to P_{in} \stackrel{\delta_{in}}{\longrightarrow} P_{i,n-1} \stackrel{\delta_{i,n-1}}{\longrightarrow} \cdots \stackrel{\delta_{i5}}{\longrightarrow} P_{i4} \stackrel{\delta_{i4}}{\longrightarrow} P_{i3} \oplus A_{\mathfrak{n}}e_i\\ 
\stackrel{\delta_{i3} \oplus \delta'_{i3}}{\longrightarrow} P_{i2} \oplus \bigoplus_{j = 1}^{\ell} A_{\mathfrak{n}}e_{\operatorname{t}(r_j)} \stackrel{\delta'_{i2}}{\longrightarrow} P_{i1} \oplus 
\bigoplus_{p \in \mathscr{P}_i} A_{\mathfrak{n}}e_{\operatorname{h}(p)} \otimes_k p
 \stackrel{\delta'_{i1}}{\longrightarrow} P_{i0} \to e_iV \to 0,
\end{multline}
where $n = N+1$ if $\underline{\dim} V = 1^{Q_0}$, and $n \leq \operatorname{max}\{3, N\}$ otherwise.
A projective resolution of $V$ is then obtained by taking a direct sum of these complexes over $i \in Q'_0$.

\begin{Theorem} \label{pd theorem}
Let $V$ be a simple $A_{\mathfrak{n}}$-module.
Then
\begin{equation*}
\operatorname{pd}_{A_{\mathfrak{n}}}V \ \left\{ \begin{array}{ll} = N+1 & \text{ if $\underline{\dim} V = 1^{Q_0}$}\\ \leq \operatorname{max}\{3, N \} & \text{ otherwise} \end{array} \right.
\end{equation*}
\end{Theorem}

\begin{proof}
Since the field $k$ is algebraically closed and $\mathfrak{n}_{\mathfrak{n}}$ is a maximal ideal, the generators for a minimal generating set of $\mathfrak{n}_{\mathfrak{n}}$ are of the form $\bar{s} - \tilde{s}$, where $s$ is a cycle in $A_{\mathfrak{n}}$, $\bar{s} \in S_{\mathfrak{n}}$ is its corresponding monomial, and $\tilde{s} \in k$ is the scalar representing $s$.
Thus, the Koszul-like connecting maps (\ref{cycle part}) cover the relations between the annihilators of $V$ in $S_{\mathfrak{n}}$, namely $\mathfrak{n}_{\mathfrak{n}}$.

Suppose $\underline{\dim}V \not = 1^{Q_0}$.
Then for each $i \in Q'_0$, $\mathscr{P}_i \not = \varnothing$.
Consequently, 
\begin{equation*}
\sigma_i - \tilde{\sigma}_i e_i = \sigma_i \in \, _{A_{\mathfrak{n}}} \! \left\langle \mathscr{P}_i \right\rangle.
\end{equation*}
Indeed, choose a path $ap \in \mathscr{P}_i$, with $a \in Q_1$.
Since $ap$ is in $\mathscr{P}_i$, we have $\tilde{e}_{\operatorname{h}(a)} = 0 \not = \tilde{p}$.
In particular, since $\tilde{e}_i \not = 0$, $ap$ itself is not a unit cycle.
Let $r \in e_{\operatorname{t}(a)}A_{\mathfrak{n}}e_{\operatorname{h}(a)}$ be a path for which $ra$ is a unit cycle.
Then 
\begin{equation*}
\sigma_i = p^{-1}p \sigma_i = p^{-1}\sigma_{\operatorname{h}(p)}p = p^{-1}rap \in \, _{A_{\mathfrak{n}}} \! \left\langle ap \right\rangle.
\end{equation*}

Furthermore, all commutation relations between $\sigma_i$ and other cycles are homotopy relations.
Thus, to construct a minimimal projective resolution of $e_iV$, it suffices to use $T_m$ rather than $T_m^{\sigma}$ in (\ref{Pm}).

Therefore, for any simple $A_{\mathfrak{n}}$-module $V$, (\ref{proj resolution}) is a projective resolution by Theorem \ref{main 2}.
\end{proof}

\begin{Theorem} \label{main corollary}
Let $P$ be a convex $2N$-gon, let $\Sigma$ be the surface obtained by identifying the opposite sides, and vertices, of $P$, and let $A = kQ/\ker \eta$ be a geodesic ghor algebra on $\Sigma$ with center $R$ and cycle algebra $S$.
Then for each maximal ideal $\mathfrak{n} \in \operatorname{Max}S$, the global dimension of the cyclic localization $A_{\mathfrak{n}}$ is bounded above by the Krull dimensions of $R$ and $S$,
\begin{equation*}
\operatorname{gldim} A_{\mathfrak{n}} \leq \dim R = \dim S = N+1,
\end{equation*}
with equality if, and only if for $N \geq 3$, $\mathfrak{n}$ is in the noetherian locus $U_{S/R} \subset \operatorname{Max}S$ of $R$.
\end{Theorem}

\begin{proof}
Follows from Proposition \ref{main0} and Theorem \ref{pd theorem}.
\end{proof}

\textbf{Acknowledgments.}
The first author was supported by a Royal Society Wolfson Fellowship RSWF/R1/180004 and the Programme Grant (EP/W007509/1).
The second author was supported by the Austrian Science Fund (grant P 34854).

\bibliographystyle{acm}
\bibliography{BaurBeil}

\end{document}